\documentclass[reqno,11pt]{amsart}
\usepackage{graphics, graphicx}
\usepackage{amsmath,fullpage,color}
\newtheorem{lemma}{Lemma}[section]
\newtheorem{theorem}{Theorem}[section]

\newtheorem{definition}{Definition}[section]
\newtheorem{remark}{Remark}[section]

\numberwithin{equation}{section} \numberwithin{theorem}{section}
\numberwithin{example}{section} \numberwithin{remark}{section}
\numberwithin{figure}{section} \numberwithin{algorithm}{section}

\def\umean{\overline{\vu}}

\def\mR{{\mathbb R}}
\def\mS{{\mathbb S}^2}
\def\ssperp{{\scriptstyle{\perp}}}
\def\Deltainv {\Delta^{\!\scriptstyle{-1}}}
\def\vu{{\mathbf u}}
\def\vup{\vu^\ssperp}
\def\vuirr{\vu_\textnormal{irr}}\def\vuinc{\vu_\textnormal{inc}}
\def\vv{{\mathbf v}}
\def\lgt{\phi}
\def\lat{\theta}

\def\ep{\varepsilon}
\def\pa{\partial}
\def\va{{\mathbf a}}
\def\ve{{\mathbf e}}
\def\vet{{\mathbf e_\lat}}
\def\vep{{\mathbf e_\lgt}}
\def\bpm{\begin{pmatrix}}
\def\epm{\end{pmatrix}}
\def\pap{\pa_{\lgt}}
\def\pat{\pa_{\lat}}
\def\be{\begin{equation}}
\def\ee{\end{equation}}
\def\dv{\textnormal{\,div\,}}
\def\curl{\textnormal{\,curl\,}}
\def\grad{\nabla}%{\textnormal{\,grad\,}}
\def\rgrad{\nabla^{\ssperp}}%{\textnormal{\,grad}^\ssperp}
\def\pt{\pa_t}
\def\id{\textnormal{id}}
\def\cL{{\mathcal L}}

\def\cP{\pcL}%{{\mathcal P}}
\def\nll{\textnormal{null}}
\def\pcL{{\displaystyle\operatorname*{\displaystyle\Pi}_{\scriptscriptstyle\nll\{\!\cL\!\}}}}

\newcommand\rev[1]{#1}
\title[Euler equations on a fast rotating sphere]{Euler equations on a fast rotating sphere\\--- time-averages and zonal flows}
\author[Bin Cheng]{Bin Cheng}
\author[Alex Mahalov]{Alex Mahalov}
\address{\newline
        School of Mathematical and Statistical Sciences\newline
    Arizona State University, Wexler Hall (PSA)\newline
    Tempe, Arizona\quad 85287-1804\quad USA}
\email[Bin Cheng]{cheng@math.asu.edu}
\email[Alex Mahalov]{Mahalov@asu.edu}

\date{\today}
\begin{document} 
\begin{abstract}Motivated by recent studies in geophysical and planetary sciences, we investigate the PDE-analytical aspects of time-averages for barotropic, inviscid flows   on a fast rotating sphere $\mS$. Of particular interests are the incompressible Euler equations.  We prove that the finite-time-average of the solution stays close to a subspace of \emph{longitude-independent  zonal flows}. The intial data can be arbitrarily far away from this subspace. Meridional variation of the Coriolis parameter underlies this phenomenon.   Our proofs use Riemannian geometric tools, in particular the Hodge Theory.\end{abstract}

\maketitle

\section{\bf \bf Introduction}
Recent studies have seen increasing understandings of \emph{global} characteristics of geophysical flows on Earth and giant planets in the Solar System. Simulations and observations have persistently shown that coherent anisotropy favoring \emph{zonal flows} appears ubiquitously in    planet scale circulations. 
For a partial list of computational results, we mention \cite{Galperin:Huang:2004} for 3D models, \cite{Vallis:1993, Nozawa:1997, Huang:Galperin:2001, Galperin:2002, zonal:beta} for 2D models, and references therein. These highly resolved, eddy-permitting simulations are made possible by rapid developements of high performance computing.
On the other hand, we have observed  zonal flow patterns (bands and jets) on giant planets for hundreds of years, which has attracted considerable interests recently thanks to spacecraft missions and the launch of the Hubble Space Telescope (e.g. \cite{Jupiter:HST}, \cite{Cassini}). Figure 1 shows a composite view of the banded strucure of Jovian atmosphere captured by the Cassini spacecraft (\cite{Jupiter:link}). There are also observational data in the oceans on Earth showing persistent zonal flow patterns (e.g. \cite{Pacific:1,Pacific:2, Maximenko}).

\begin{figure}[htbp]
\centering\includegraphics[width=0.6\textwidth]{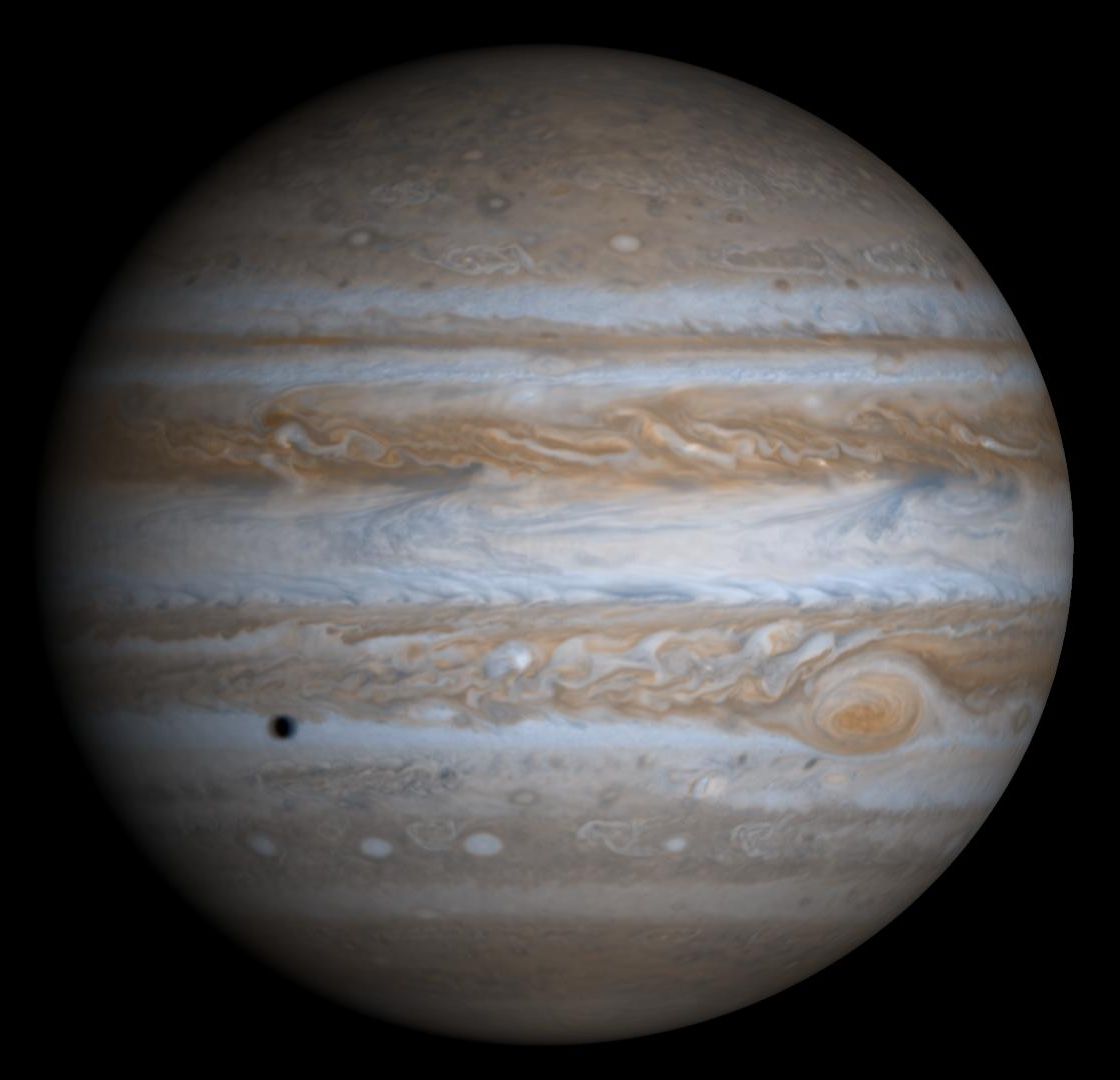}
\caption{This true-color simulated view of Jupiter is composed of 4 images taken by NASA's Cassini spacecraft on December 7, 2000. Credit: NASA/JPL/University of Arizona \cite{Jupiter:link}.}
\label{fig1}
\end{figure}

\begin{remark}\label{remark:intro}
In the existing literature, a necessary process for the zonal flow pattern to emerge is  averaging of data and/or simulations over a period of time usually decades long.  Also, a number of heuristic arguments (e.g. \cite{zonal:beta}) are made pointing to   the north-south gradient of Coriolis parameter as the underlying machanism, even in the absence of temperature/density gradient and vertical variability. 
 \end{remark}

To this end, we study   inviscid, barotropic geophysical flows on a unit sphere $\mS$ centered at the origin of $\mR^3$ and fast rotating about the $z$-axis with constant angular velocity.  Let vector field $\vu(t,q)$,  tangent to $\mS$ at every point $q\in\mS$, denote the fluid velocity {\it relative} to this rotating frame. Throughout this paper, we represent any point $q\in \mS$ either by its relative-to-the-frame cartesian coordinates $(x,y,z)$ or its relative-to-the-frame spherical coordinates $(\theta,\phi)$  with $\theta$ being the colatitude and $\phi$ the longitude\footnote{ At $z=1$, we fix $\theta=0$ but $\phi$ can be arbitrary.  Such singularity issue does not occur for cartesian coordinates $x,y,z\in C^\infty(\mS)$.}. Also, let $\vep$ be the unit vector in the zonal direction of increasing longitude and $\vet$ be the unit vector in the meridional direction of increasing colatitude.

In this study, we focus on a canonical PDE system: the incompressible Euler equations  under the Coriolis force (\cite{Ar:Kh:Hydro, Ch:Ma:Fluid, Pedlosky:Geo}),\be\label{Euler:vec}\pt\vu+\nabla_\vu\vu+\grad P={z\over\ep}\vup,\qquad \dv\vu=0\ee 
where constant $\ep$, called the Rossby number, scales like the frequency of the frame's rotation (usually $0.01\sim0.1$ at a global scale), and $\vup$ denotes a counterclockwise $\pi/2$-rotation of $\vu$ on $\mS$. Cartesian coordinate $z$  indicates how the Coriolis parameter  varies along the meridional direction. Note that the Coriolis force   is not uniformly large and actually vanishes on the equator. It is the large \emph{gradient} of the Coriolis parameter that drives the zonal flow patterns.

Note that results have been established concerning  solution regularity of the above systems and related ones in the fast rotating regime with $\ep\ll1$. Please refer to \cite{BMN:RSW, BMN:NS, Goncharov} and references therein for further discussion.

Our theoretical investigation is then focused on the fast rotating regime with $\ep\ll1$ and in particular,  the nature of the {\it time-averages} of $\vu$: \be\label{def:umean}\umean(T,\cdot):={1\over T}\int_{0}^{T}\vu(t,\cdot)\,dt\ee for positive times $T$.

The main result is stated as following.
\begin{theorem}\label{thm:zonal}Consider the incompressible Euler equations \eqref{Euler:vec}   on $\mS$ with div-free initial data  $\vu(0,\cdot)\in H^k(\mS)$ for $k\ge3$. Define the time-averaged  flow $\umean$ as in \eqref{def:umean}. Then, there exist  $\ep$-independent constants $C_0, T_0$ s.t. for any given $T\in[0,T_0/\|\vu_0\|_{H^k}]$, there exist a function $f(\cdot):[-1,1]\mapsto\mR$ and a universal constant $C$ s.t. 
\be\label{estimate:thm}\begin{split}&\left\|\umean(T,x,y,z)-\rgrad f(z)\right\|_{H^{k-3}(\mS)}\le C_0\ep({M_0\over T}+M_0^2).\end{split}\ee
Here, $M_0:=\|\vu_0\|_{H^k}$ indicates the size of initial data. In spherical coordinates, the approximation $\rgrad f(z)$ is
\[\rgrad f(z)=-f'(\cos\theta)\sin\theta\,\vep.\]
 \end{theorem}

Our theoretical result proves computational and observational results in the literature mentioned at the beginning of this paper, especially Remark \ref{remark:intro}. Note that our result shows that the zonal-flow pattern becomes prominant with decreasing Rossby number $\ep\searrow0$. In other words, the time-averaged flow $\umean$ is only $O(\ep)$ away from a very restricted subspace consisting of {\bf longitude-independent zonal flows}. The initial data, on the other hand, do not need any filtering and can be arbitrarily far away from that subspace of zonal flows. Our proofs below will suggest that such unique pattern   is essentially due to the Coriolis parameter $z/\ep$ that varies meridionally from the strongest at the poles to zero on the equator.  

Rossby number $\ep$ is typically at magnitude 0.01$\sim$0.1 for Earth oceans, which results in a time scale of magnitude 10$\sim$100 Earth days according to Theorem \ref{thm:zonal}. This suggests that zonal flow patterns can occur at time scales  far below those used in the literature.  In fact, the Rossby number is even smaller for giant planets, leading to the direct observability of banded structures.

Our result for rotating incompressible Euler equations  uses the abstract framework of the following lemma.
\begin{lemma}\label{lemma:abstract}Consider time-dependent equation\[\vu_t={1\over\ep}\cL[\vu]+f(t,q)\]over certain spatial domain $\Omega$.
Here, $\ep>0$ is a scaling constant, $f(t,q)$  a source term and $\cL$   a linear operator independent of time. 
Let operator $\cP$ denote (some) projection onto the 
null space of $\cL$. Assume a priori $\vu, f(t,q), \cL[\vu],\cP[\vu]$ have enough regularity as needed. 

 Then, under the assumption\be\label{bound}\|\vu-\cP\vu\|_{H^{k_1}(\Omega)}\le C \|\cL[\vu]\|_{H^{k_2}(\Omega)} \ee
for some constant $C$,
holds true the following estimate on the \underline{time-average} of $\vu$,
\[\Big\|{1\over T}\int_0^T\vu\,dt -{1\over T}\int_0^T\cP\vu\,dt\Big\|_{H^{k_1}}\le \ep C\left({2M\over T}+M'\right)\]where   constants $M:=\max_{t\in[0,T]}\|\vu(t,\cdot)\|_{H^{k_2}}$ and $M':=\max_{t\in[0,T]}\|f(t,\cdot)\|_{H^{k_2}}$.\end{lemma}

\begin{remark}\label{remark:finite_D}The key hypothesis \eqref{bound} is automatically true in a finite-dimensional space if $\vu$ is a vector in $\mR^n$, $\cL$ a linear transform $\mR^n\mapsto\mR^n$ and $\pcL$ the $l^2$-projection onto $\nll\{\cL\}$. In such case, hypothesis \eqref{bound}, with the norms understood as $l^2$ norm on both sides, amounts to the boundedness of $\cL^{-1}:\textnormal{image}\{\cL\}\mapsto \mR^n/\nll\{\cL\}$  \end{remark}
\begin{proof}[Proof of Lemma \ref{lemma:abstract}]First, transform the original equation into
\[\vu_t={1\over\ep}\cL[\vu-\cP\vu]+f\]and apply time-averaging ${1\over T}\int_0^T\cdot\, dt$ on both sides
\[{1\over T}\left(\vu(T,\cdot)-\vu(0,\cdot)\right)={1\over\ep T}\int_0^T\cL[\vu-\cP\vu]dt+{1\over T}\int_0^T f(t,\cdot)dt\] Since all necessary regularities were assumed available and $\cL$ was assumed to be linear and independent of time, we argue that $\int_0^T\cdot\, dt$ and $\cL[\cdot]$ commute, so that the above equation becomes
\[{1\over T}\left(\vu(T,\cdot)-\vu(0,\cdot)\right)={1\over\ep}\cL\Big[{1\over T}\int_0^T \vu\,dt -{1\over T}\int_0^T\cP\vu\,dt\Big]+{1\over T}\int_0^T f(t,\cdot)\,dt.\]
Due to the factor ${1\over\ep}$ in the first term on the RHS, we have
\[\left\|\cL\Big[{1\over T}\int_0^T \vu\,dt -{1\over T}\int_0^T\cP\vu\,dt\Big]\right\|_{H^{k_2}}\le \ep\left({2M\over T}+M'\right).\] 
 Finally, apply estimate \eqref{bound} to arrive at the conclusion. 
\end{proof}

Note that, the constant $M$ used in the above lemma depends on size of the solution up until time $T$ and is {\it not} necessarily independent of $\ep$. A priori estimates uniform in $\ep$ are therefore in order. The proof requires considerations beyond the well established energy methods, which will be explained in Appendix B.

\begin{proof}[Proof of Main Theorem \ref{thm:zonal}]
Having Lemma \ref{lemma:abstract} and $\ep$-independent estimates in Theorem \ref{thm:uniform}, it suffices to study properties of properly defined operators $\cL$ (c.f. Definition \ref{definition:cL}) and $\cP$ (c.f. Lemma \ref{lemma:pcL}), and to finally prove estimate \eqref{bound} (c.f. Theorem \ref{thm:cL}).\end{proof}

The rest of the paper is organized as following. We start Section 2 with describing a version of the Hodge decomposition in terms of differential operators on $\mS$. The definitions of these operators are given the Appendix.  An elliptic operator $\cL$ that plays the same role as the $\cL$ in Lemma \ref{lemma:abstract} is defined by the end of Section 2. In Section 3, we characterize the null space of $\cL$, identifying $\nll\{\cL\}$ as the space of longitude-independent zonal flows (c.f. Lemma \ref{null:cL:zonal}). We also define the projection operator $\pcL$ and its complement.  In   Section 4, we obtain Sobolev-type estimates, in particual \eqref{bound},  regarding $\cL$ and $\pcL$ using the spherical coordinates and spherical harmonics. 
In Appendix A, we give the rigorous definitions of differential operators on surfaces such as $\mS$ and prove related properties. It is necessary to adopt coordinate-independent differential geometric tools  since any global coordinate system on $\mS$ is bound to have singularity issues. On the other hand, one can formally use spherical coordinates as well as cartesian coordinates for most of the arguments presented in this paper, knowing their validity is justified. In Appendix B, we prove an $\ep$-independent estimates after carefully examining commutability properties of some differential-integral operators on a sphere.
\section{\bf Hodge Decomposition}

The {\bf Hodge decomposition} theorem  (\cite{Warner:Diff}, \cite{Taylor:I}) confirms that for any $k$-form $\omega$ on an oriented compact Riemannian manifold, there exist a $(k-1)$-form $\alpha$, $(k+1)$-form $\beta$ and a harmonic $k$-form $\gamma$, s.t.
\[\omega=d\alpha+\delta\beta+\gamma.\]

In particular, if the manifold is a surface in the cohomology class of $\mS$  (loosely speaking, there is no ``hole'' or "handle"), then there exist two scalar-valued functions $\Phi$ (called potential) and $\Psi$ (called stream function) such that
\be\label{Hodge:vec}\vu=\vuirr+\vuinc\mbox{ \quad where \quad }\vuirr:=\grad \Phi\mbox{ and }\vuinc:=\rgrad \Psi.\ee
Please refer to \eqref{Hodge:vec:a} in the Appendix and the discussion that leads to it.

Moreover, the decomposition satisfies
\[\curl\vuirr=\dv\vuinc=0\]and\be\label{irr:inc} \vuirr=\grad\Delta_{\mS}^{-1}\dv\vu,\quad\quad \vuinc=\rgrad\Delta_{\mS}^{-1}\curl\vu.\ee
In other words, any (square-integrable) vector field on $\mS$ can be written as superposition of an incompressible and an irrotational vector fields that are determined by \eqref{irr:inc}. Such decomposition is unique because a harmonic scalar-valued function on a sphere (and any surface in the same cohomology class) is always constant and therefore $\Delta_{\mS}^{-1}$ is unique up to a constant.

For simplicity, we will use $\Delta$ for $\Delta_{\mS}$ from here on. Also, we assume that, unless specified otherwise,
\be\label{Delta:global:mean}\Deltainv f\mbox{ always has zero global mean over }\mS.\ee

We postpone the differential-geometric definitions and properties of $\grad, \rgrad, \curl, \dv, \Delta$ on $\mS$ till the Appendix.

Now, rearrange \eqref{Euler:vec} as\[{z\over\ep} \vup- \nabla_\vu\vu=\pt\vu+\grad P.\]
Observe that on the RHS, $\pt\vu$ is incompressible and $\grad P$ is irrotational. Thus, the RHS is the unique Hodge decomposition of the LHS, which satisfies the elliptic PDEs \eqref{irr:inc}. In particular, the incompressible part $\pt\vu$ is uniquely determined by \be\label{Euler:proj}\pt\vu=\rgrad\Deltainv \curl\left({z\over\ep} \vup- \nabla_\vu\vu\right).\ee 

This is indeed an equivalent formulation of  the original incompressible Euler's equation \eqref{Euler:vec}.

In the context of Lemma \ref{lemma:abstract}, we define the following operator
\begin{definition}\label{definition:cL} For any $\vu$, not necessarily div-free, define \be\label{def:cL}\cL[\vu]:=\rgrad\Deltainv \curl(z\vup).\ee
Here, $\Deltainv $ follows the convention \eqref{Delta:global:mean}.\end{definition}

Then, \eqref{Euler:proj} can be reformulated as,
\be\label{Euler:cL}\pt\vu+\rgrad\Deltainv \curl(\nabla_\vu\vu)={1\over\ep}\cL[\vu].\ee

\section{\bf Null Space of $\cL$ and associated $L^2$-orthogonal projection}
We observe that the definition \eqref{def:cL} naturally implies a sufficient condition for $\cL[\vu]=0$  is \[\curl(z\vup)=0.\] Indeed this is also necessary by the virtue of  \eqref{elliptic}, namely $\curl\rgrad=\Delta$, so that\[\curl\cL[\vu]=\curl\rgrad\Deltainv \curl(z\vup)=\curl(z\vup).\]
Thus, for velocity field $\vu$, not necessarily div-free,\be\label{null:cL:0}\cL[\vu]=0\iff\curl(z\vup)=0.\ee

Further analysis reveals the following lemma.
\begin{lemma}\label{null:cL:zonal}(Characterization of $\nll\{\cL\}$) For div-free $\vu$ with sufficient regularity,\be\label{null:cL}\quad\cL[\vu]=0\iff \vu=\rgrad g(z)=-g'(\cos\theta)\sin\theta\vep\ee 
for some function $g:[-1,1]\mapsto \mR$. Thus, we identify $\nll\{\cL\}$, when restricted to div-free velocity fields, with  the space of \;{\bf longitude-independent zonal flows}.
\end{lemma}

\begin{proof}[Proof of Lemma \ref{null:cL:zonal}] Apply the product rule \eqref{prod} in the Appendix to  $\curl(z\vup)$,\be\label{curlzvup}\curl(z\vup)=(\grad z)\cdot\vu+z\dv\vu=(\grad z)\cdot\vu\ee
where the $\cdot$ product is given by the natural metric on $\mS$ induced from $\mR^3$ and we used the incompressibility condition $\dv\vu=0$. Therefore, by \eqref{null:cL:0} and \eqref{curlzvup},  the null space of $\cL$ is identified as
\be\label{null:cL:1}\mbox{For any div-free $\vu$ },\cL[\vu]=0\iff (\grad z)\cdot\vu=0.\ee

Since $\grad z$ is in the meridional direction, \eqref{null:cL:1} implies any $\vu$ in the null space of $\cL$ flows in the zonal direction. But there is more than that. Hodge decomposition \eqref{Hodge:vec}, \eqref{irr:inc} implies
\be\label{dv:vu:Psi}\dv\vu=0\iff\vu=\rgrad\Psi(x,y,z)\ee
with $\Psi$ being a unique scalar function with zero global mean. Combining \eqref{dv:vu:Psi} it with \eqref{null:cL:1}, we have, for any incompressible velocity field $\vu$,
\[\vu=\rgrad\Psi\in\nll\{\cL\}\iff (\grad z)\parallel(\grad\Psi).\] The condition $(\grad z)\parallel(\grad\Psi)$ implies $\Psi$ is a function of $z$ only. Thus, we arrive at the conclusion. The very last term in \eqref{null:cL} is due to the fact that, in spherical coodinates,
\[\rgrad z=-\sin\theta\vep.\]
\end{proof}

It is then easy to show the following characterization of $\pcL$, the $L^2$-orthogonal-projection operator onto $\nll\{\cL\}$.
\begin{lemma}\label{lemma:pcL} (Characterization of $\pcL$) For any div-free vector field $\vu\in L^2(\mS)$, its $L^2$-orthogonal-projection onto $\nll\{\cL\}$ satisfies
\be\label{pcL}\pcL\vu=\dfrac{\left(\oint_{C(\theta)}\vu\cdot \vep\right)\vep}{\oint_{C(\theta)}\vep\cdot\vep}={1\over2\pi\sin\theta}\left(\oint_{C(\theta)}\vu\cdot \vep\right)\vep\ee
where  $\oint_{C(\theta)}$ is the line integral along the circle $C(\theta)$ at a fixed colatitude $\theta$. \end{lemma}
Several remarks are in order. First, among all possible projection operators, we chose one that nullifies the $L^2(\mS)$ orthogonal complement of
$\nll\{\cL\}$.
Secondly, intuitively, $\pcL\vu$ at a given latitude is a uniform zonal flow equal to the mean circulation of $\vu$ at that latitude; thus, $(\id-\pcL)\vu$ is of zero circulation along the circle at a fixed latitude. Such intuition is consistent with the orthogonality condition
\[\int_{\mS}(\pcL\vu)\cdot(\id-\pcL)\vu'=0.\]
Thirdly, even though \eqref{pcL} runs into singularity at the poles $\theta=0$ and $\theta=\pi$, such singularity is removable. In fact,    apply the Stokes' theorem on the RHS of \eqref{pcL} so that, for $\theta\in(0,\pi)$,
\[\pcL\vu={1\over2\pi\sin\theta}\left(\int\!\!\!\!\int_{interior\; of \;C(\theta)}\curl\vu\right)\vep\]and by taking the limit as $\theta\to0+$ and $\theta\to\pi-$, we obtain\[\pcL\vu\Big|_{poles}={\bf 0}.\]

In terms of the stream function, for any div-free velocity field $\vu=\rgrad\Psi$ which amounts to $\vu=(\pa_\theta\Psi)\vet^\ssperp+\left({\pa_\phi\Psi\over\sin\theta}\right)\vep^\ssperp$ with removable singularity at the poles,
\be\label{proj:Psi}\pcL(\rgrad\Psi)={1\over2\pi}\left(\int_0^{2\pi}\pa_\theta\Psi\,d\phi\right)\vep=\rgrad
\left({1\over2\pi}\int_0^{2\pi}\Psi\,d\phi\right).\ee
In other words, $\pcL$ maps the stream function to its zonal means.

\section{\bf Key Estimates}
This section is dedictated to proving an estimate similar to \eqref{bound}. A convinient tool in studying Sobolev norms of functions on $\mS$ is the   spherical harmonics. 

To this end, introduce the spherical harmonc of   degree $l$ and order $m$,
\[Y_l^m(\phi,\theta)=N_l^me^{im\phi}Q_l^m(\cos\theta),\quad\mbox{ for }l=0,1,2,...\mbox{ and }m=-l,-l+1,...,0,...,l-1, l \]where the normalizing constant $N_l^m=\sqrt{{(2l+1)\over4\pi}{(l-m)!\over(l+m)!}}$ so that $\|Y_l^m\|_{L^2(\mS)}=1$. It satisfies the eigenvalue problem
\be\label{eig:Y}\Delta Y_l^m=-l(l+1)Y_l^m\ee and orthonomal condition\[\langle{Y_l^m}, {Y_{l'}^{m'}}\rangle_{L^2(\mS)}=\delta_{ll'}\delta_{mm'}.\]
Here and below, \[\langle f,g\rangle_{L^2(\mS)}:=\int_{\mS}\overline{f}g
\,d\Omega\] with $d\Omega$ being the area element of $\mS$ and locally equals $\sin\theta d\theta d\phi$.
The assoicated Legendre polynomial $Q_l^m(z)$ satisfies the general Legendre equation\[{d\over dz}\left(\sqrt{1-z^2}{d\over dz}Q_l^m(z)\right)+\left(l(l+1)-{m^2\over1-z^2}\right)Q_l^m=0\]and can be expressed via the Rodrigues's formula,
\[Q_l^m(z)={1\over 2^ll!}(1-z^2)^{m\over2}{d^{m+l}\over dz^{m+l}}(1-z^2)^l.\] 

In order to estimate the Sobolve norms (esp. $H^k$ norms) of $Y_l^m$, we take the $L^2(\mS)$ inner product of \eqref{eig:Y} with $Y$  (omitting indices for simplicity), invoke Green's identity \eqref{Delta:Green} to calculate
\[\label{IBP}\begin{split}l(l+1)=&l(l+1)\langle Y,Y\rangle_{L^2(\mS)}\\=&-\langle {Y},\Delta Y\rangle_{L^(\mS)}=\langle\grad Y,\grad Y\rangle_{L^2(\mS)}\end{split}\]
which implies
\[Y\in H^1(\mS)\quad\mbox{ so inductively }\quad Y\in H^k(\mS),\;\;k\ge0.\]
As a matter of fact, a little more rigor is needed in defining Sobolev norms on a manifold, but we will skip the technical details and only use the fact that
\[\label{harmonic:Hk}C'_k(1+l)^k\le\|Y_l^m\|_{H^k(\mS)}\le C_k(1+l)^k,\quad\mbox{ for }\quad l\ge0,\;\;k\ge0\]

This estimate allows us to use the following definition, among many other equivalent definitions (\cite{Taylor:I}), of the $H^k$ norm of a scalar-valued function $\Psi$ defined on $\mS$.  \begin{definition}For a scalar function $\Psi$ with $\int_{\mS}\Psi=0$ and series expansion
\be\label{Psi:Fourier}\Psi=\sum_{l=1}^\infty\sum_{m=-l}^l \psi_{l}^{m}Y_l^m,\quad\quad\mbox{ where } \psi_{l}^{m}=\langle{\Psi}, Y_l^m\rangle_{L^2(\mS)},\ee
define it $H^k$ norms, among other equivalent versions, as
\be\label{def:Hk}\|\Psi\|_{H^k}:=\sqrt{\sum_{l=1}^\infty\sum_{m=-l}^l (1+l)^{2k}|\psi_{l}^{m}|^2}.\ee
\end{definition}
\begin{remark}Here and below, we always start  series and sums with $l=1$ and assume $\psi_0^0=\int_{\mS}\Psi=0$.
\end{remark}
Consequently, we define $H^k$ norms for $\vu$.
\begin{definition}For a vector field $\vu$ with Hodge Decomposition \[\vu=\grad\Phi+\rgrad \Psi\mbox{ \quad with $\int_{\mS}\Phi=\int_{\mS}\Psi=0$},\]we define its $H^{k}$ norm, among other equivalent versions, as  
\be\label{def:vuHk:general}\|\vu\|_{H^k}:=\sqrt{\|\Phi\|^2_{H^{k+1}}+\|\Psi\|^2_{H^{k+1}}}.\ee

In particular, if $\vu$ is div-free with $\vu=\rgrad\Psi$ and $\int_{\mS}\Psi=0$, then
\be\label{def:vuHk}\left\|\rgrad\Psi\right\|_{H^k}=\sqrt{\sum_{l=1}^\infty\sum_{m=-l}^l (1+l)^{2(k+1)}|\psi_{l}^{m}|^2}.\ee\end{definition}
\begin{remark}Here and below, we always choose $\Phi$, $\Psi$ with zero global mean such that 
the above definition is consistent with $\|{\bf 0}\|_{H^k}=0$.\end{remark}
\begin{remark}\label{remark:Hk:orthogonal}
 Apparently, under above definition, div-free and curl-free vector fields are $H^k$-orthogonal.
\end{remark}

We now characterize operator $\cL$ using the spherical harmonics. Let incompressible velocity field \[\vu=\rgrad\Psi.\]Then, by definition \eqref{def:cL} and identity 
\eqref{curlzvup}
\[ \cL[\rgrad\Psi]=\rgrad\Deltainv \left(\grad z\cdot\rgrad\Psi\right)\]It is easy to verify that, in spherical coordinates,\[\grad z=-\sin\theta\,\vet\mbox{\quad and \quad} \rgrad\Psi=(\pa_\theta\Psi)\vet^\ssperp+\left({\pa_\phi\Psi\over\sin\theta}\right)\vep^\ssperp .\]Thus, combining the three equalities above, we obtain
\be\label{cL:stream}\cL[\rgrad\Psi ]=\rgrad\Deltainv \pa_\phi\Psi.\ee
\begin{lemma}(Spherical-harmonic representation of $\cL$.) For a scalar function $\Psi$ with a series expansion \eqref{Psi:Fourier}, the identity \eqref{cL:stream} leads to
\be\label{cL:Fourier}\cL[\rgrad\Psi]=\rgrad\sum_{l=1}^\infty\sum_{\substack{m=-l\\ m\ne0}}^l{-im\over l(l+1)}\psi_{l}^{m}Y_l^m.\ee\end{lemma}
Here, we used the fact that $\pa_\phi Y_l^m=imY_l^m$ and $\Deltainv Y_l^m=-{1\over l(l+1)}Y_l^m$ for $l\ge1$.  We exclude $l=0$ from the series due to $\psi_0^0=\int_{\mS}\Psi=0$. We also exclude $m=0$ since it doesn't contribute to \eqref{cL:Fourier} anyway.

We now use spherical harmonics to characterize the
projection operator $\pcL$ given in \eqref{proj:Psi}. It follows from \eqref{cL:Fourier} that
\[\vu\in\nll\{\cL\}\mbox{\quad iff \quad}\vu=\rgrad\sum_{l=1}^\infty\sum_{m=0}\psi_{l}^{m}Y_l^m,\]
which is consistent with Lemma \ref{null:cL:zonal} since $Y_l^0$ is a function of $\theta$ only.

Thus, the only modes that survive $\pcL$ are those with $m=0$ since  $\pcL$ is an $L^2$-orthogonal projection and \[\int_{\mS}\rgrad\overline{ Y_{l}^{m}}\cdot\rgrad{Y_{l'}^{m'}}=-\int_{\mS}\Delta\overline{Y_{l}^{m}}  Y_{l'}^{m'}=l(l+1)\delta_{ll'}\delta_{mm'}.\] 
\begin{lemma}(Spherical-harmonic representation of $\pcL$.) For a scalar function $\Psi$ with a series expansion \eqref{Psi:Fourier}, 
\begin{align}\label{proj}\pcL(\rgrad\Psi)=&\rgrad\left(\sum_{l=1}^\infty\sum_{m=0}\psi_{l}^{m}Y_l^m\right)\\
\label{c_proj}(\id-\pcL)(\rgrad\Psi)=&\rgrad\left(\sum_{l=1}^\infty\sum_{\substack{m=-l\\ m\ne0}}^{l}\psi_{l}^{m}Y_l^m\right)\end{align}\end{lemma}
Note that the above 2 equations can also be derived from \eqref{proj:Psi} together with the fact that
\[{1\over2\pi}\int_0^{2\pi}Y_l^m(\phi,\theta)d\phi=\delta_{m0}.\]

Combining \eqref{cL:Fourier} with \eqref{c_proj} and using the absence of $m=0$ modes from both series, we induce that, when $\cL$ is restricted to the image of $(\id-\pcL)$, its null space is trivial and its inverse is ``bounded'' (as noted in Remark \ref{remark:finite_D}, this is
automatically true for linear transform $\cL:\mR^n\mapsto\mR^n$).

More precisely,

\begin{theorem}\label{thm:cL}For any div-free vector field $\vu\in H^k(\mS)$ and $k\ge0$,
\[\begin{split}&\Big\|\vu-\pcL\vu\Big\|_{H^{k}}\\\le &\Big\|\cL\big[\vu-\pcL\vu\big]\Big\|_{H^{k+2}}=\left\| \cL[\vu]\right\|_{H^{k+2}}\end{split}\]\end{theorem}
\begin{proof} Consider the stream function $\Psi$ so that $\vu=\rgrad\Psi$. Combining  \eqref{def:vuHk} and \eqref{c_proj}, we
obtain
\[\Big\|\vu-\pcL\vu\Big\|_{H^{k}}=\sqrt{\sum_{l=1}^\infty\sum_{\substack{m=-l\\ m\ne0}}^l (1+l)^{2(k+1)}|\psi_{l}^{m}|^2}.\]
Combining \eqref{def:vuHk} and \eqref{cL:Fourier}, we obtain
\[\|\cL[\vu]\|_{H^{k+2}}=\sqrt{\sum_{l=1}^\infty\sum_{\substack{m=-l\\ m\ne0}}^l (1+l)^{2(k+3)}\left|{m\psi_{l}^{m}\over l(l+1)}\right|^2}.\]
The key observation here is that $m=0$ modes are absent in both series; thus, by a simple inequality
\[(1+l)^{2(k+1)}\le(1+l)^{2(k+3)}\left|{m\over l(l+1)}\right|^2\quad\mbox{ for }l\ge1,\;|m|\ge1,\]
we arrive at the conclusion!
\end{proof}

\section{\bf Appendix A: Preparation in Differential Geometry}
Let $M$ denote a 2-dimensional, compact, Riemmanian manifold without boundary, typically the unit sphere $M=S^{3-1}$  endowed with metric $g$ induced from the embedding Euclidean space $\mR^3$. Let $p\in M$ denote a point with local coordinates $(p_1,p_2)$. Any vector field $\vu$ in the tagent bundle $TM$ is identified with   a field of directional differential operator which is written in local coordinates as\[\vu=\sum_ia^{i}{\pa\over\pa p_i}.\]We use the notation
\[\nabla_{\vu} f:=\sum_i a^i {\pa f\over\pa p_i}\]
to denote the directional derivative of a scalar-valued function $f$ in the direction of $\vu$. Using the orthogonal projection $\mbox{Proj}_{T\mR^3\to TM}$ induced by the Euclidean metric of $\mR^3$, we define the covariant derivative of a vector field $\vv\in TM$ along another vector field $\vu\in TM$, \be\label{def:cov}\nabla_{\vu}\vv:=\mbox{Proj}_{T\mR^3\to TM}\sum_{i=1}^3(\nabla_\vu v_i)\ve_i.\ee
Here, $\vv$ is expressed in an orthonomal basis of $\mR^3$ as $\vv=v_1\ve_1+v_2\ve_2+v_3\ve_3$.

The metric $g$ is identified with a $(0,2)$ tensor, simply put, an $2\times 2$ matrix $(g_{ij})_{2\times 2}$ in local coordinates. Thus, the vector inner product follows
\[\rev{g({\pa\over\pa{p_i}},{\pa\over\pa{p_j}})}=g_{ij}\quad\mbox{ for }1\le i,j\le 2.\]

\subsection{Hodge Theory}(\cite{Warner:Diff,Taylor:I})

The Hodge *-operator, defined in an {\it orthonormal basis}\footnote{The existence of such basis is guaranteed by the Gram-Schmidt orthogonalization process.} ${\pa\over\pa p_1},\,{\pa\over\pa p_2}$, satisfies
\[*dp_1=dp_2,\;\;*dp_2=-dp_1,\;\;*1=dp_1\wedge dp_2,\;\;*(dp_1\wedge dp_2)=1.\]

Using the Hodge star operator, we define the co-differential for any $k$-forms $\alpha$ in an $n$-dimensional manifold,
\[\mbox{\bf codifferential}:\quad\delta\alpha:=(-1)^k*^{-1}d*\alpha=(-1)^{n(k+1)+1}*d*\rev{\alpha},\]
and in particular, for $n=2$,
\[\delta\alpha=-*d*\rev{\alpha}.\]In the case when the Riemannian manifold $M$ has no boundary, the codifferential is the adjoint of exterior differential w.r.t. $L^2(M)$ inner product induced by the given metric $g$,
\be\label{dual}\langle d\alpha,\beta\rangle_{\rev{L^2(M,g)}}=\langle \alpha,\delta\beta \rangle_{\rev{L^2(M,g)}}.\ee

The Hodge Laplacian (a.k.a. Laplace-Beltrami operator and Laplace-de Rham operator) is then defined by
\be\label{def:Hodge_L}\Delta_H:=d\delta+\delta d.\ee
In particular, for a scalar-valued function $f$  in a local basis $\left\{{\pa\over\pa{p_i}}\right\}$   with metric g, it is identified as \[\Delta_H f=-{1\over\sqrt{|g|}}\sum_{i,j} \pa_i(\sqrt{|g|}g^{ij}\pa_j f)\]where $(g^{ij})$ is the matrix inverse of $(g_{ij})$. Thus, on a surface $M$, the Hodge Laplacian $\Delta_H$ defined in\eqref{def:Hodge_L} amounts to  -1 times the surface Laplacian $\Delta_M$. In particular, if $M$ is a two-dimensional surface, then
\be\label{surface:L}\mbox{for scalar function }f,\quad\Delta_M f=-\delta df=*d*df\ee
since $\delta f=0$ on a two-dimensional manifold.

For now on, we will use $\Delta$ for $\Delta_M$.

The {\bf Hodge decomposition} theorem in its most general form states that for any $k$-form $\omega$ on an oriented compact Riemannian manifold, there exist a $(k-1)$-form $\alpha$, $(k+1)$-form $\beta$ and a harmonic $k$-form $\gamma$ satisfying $\Delta_H\gamma=0$, s.t.
\[\omega=d\alpha+\delta\beta+\gamma.\]
In particular, for any 1-form $\omega$ on a 2-dimensional manifold with the 1st Betti number $0$ (loosely speaking, there is no ``holes''), there exist two scalar-valued functions $\Phi,\,\Psi$ such that
\be\label{Hodge:1form}\omega=d\Phi+\delta(*\Psi)=d\Phi-*d\Psi.\ee
Here, we used the Hodge theory to equate the dimension of the space of harmonic $k$-forms on $M$ with the $k$-th Betti number of $M$. For the cohomology class containing the unit sphere $S^{3-1}$, the 0th, 1st and 2nd Betti numbers are respectively $1,0,1$.

\subsection{In Connection With Vector Fields}
Let  $\flat$ and $\sharp$ denote the musical isomorphism between $TM$ and $T^*M$ induced by the given metric $g$, i.e. for vector field $\va=\sum_ia^i\pa_i$ and covector field (1-form)
$A=\sum_iA_idp^i$
\[\begin{split}\mbox{index-lowering,\qquad\qquad}&(\sum_ia^i\pa_i)^\flat=\sum_{i,j}g_{ij}a^idp_j\\
   \mbox{index-raising,\qquad\qquad}&(\sum_iA_idp^i)^\sharp=\sum_{i,j}g^{ij}A_i\pa_j  
\end{split}\]

In connection with the metric $g$, for vector fields $\vu,\vv$,
\[*(\vu\cdot\vv)=\vu^\flat\wedge(*\vv^\flat)\]
so that
consistently,
\be\label{dot:product}\langle\vu,\vv\rangle_{L^2(M,g)}=\langle\vu^\flat,\vv^\flat\rangle_{L^2(M,g)}.\ee

In a 2-dimensional Riemannin manifold, the 
divergence and curl of a vector field $\vu\in TM$ are then defined as scalar-valued functions\footnote{Here and below,  0-forms  are identified with  scalar-valued functions.} on $M$,
\begin{align}\label{def:div} {\mbox{\bf divergence},\quad}&\dv\vu:=-\delta(\vu^\flat)=*d*(\vu^\flat)\\
\label{def:curl}\mbox{\bf curl},\quad&\curl\vu:=-*d(\vu^\flat)\end{align}

For a scalar field $f$, we define gradient and its $\pi/2$ rotation as
\begin{align}\label{def:grad}{\mbox{\bf gradient},\quad}&\grad f:=(df)^\sharp\\
\label{def:rgrad}\mbox{\bf rotated gradient},\quad&\rgrad  f:=(\delta(*f))^\sharp=-(*df)^\sharp\end{align}
We also define the counterclockwise $\pi/2$ rotation operator ${}^\ssperp$ acting on a vector field as
\be\label{def:perp}\vup:=-(*\vu^\flat)^\sharp\ee
so that, consistantly,\[\rgrad f=(\grad f)^\ssperp,\] 
and \[\dv \vu=\curl\vup.\]

It is then easy to use these definitions to verify the following properties, for scalar function $f$ on a surface,
\be\label{cancell}\curl\grad f =\dv\rgrad f=0\ee
due to $dd=0$ and $\delta\delta=0$; and\be\label{elliptic}\dv\grad f=\curl\rgrad f=-\delta d\, f\ee
with $-\delta d$ being the classical surface Laplacian $\Delta$ (c.f. \eqref{surface:L}).

To this end, the vector-field version of Hodge decomposition \eqref{Hodge:1form} becomes
\be\label{Hodge:vec:a}\vu=\grad \Phi+\rgrad \Psi.\ee  We note that, by the virtue of \eqref{elliptic}, the decomposition satisfies
\[ \dv\vu=\Delta\Phi,\quad\quad \curl\vu=\Delta\Psi.\]

We finally establish a version of the Green's identity and a version of the product rule on Riemannian manifolds. First, the duality relation \eqref{dual} together with \eqref{surface:L} and \eqref{dot:product} implies
\be\label{Delta:Green}\langle f,\Delta g\rangle_{L^2(M)}=-\langle\grad f,\grad g\rangle_{L^2(M,g)}.\ee
Secondly, as a consequence of  the product rule for differential $d$ acting on wedge product,
{for scalar function $z$ and vector field $\vu$}, we have
\be\label{prod}\curl(z\vup)=\dv(z\vu)=\grad z\cdot\vu+z\dv\vu.\ee

\subsection{Local Expression in Terms of Spherical Coordinates for $M=S^{2}$}Let $\lgt$ denote the logitude and $\lat$ the colatitude of a point on a sphere. 
Let $\vep,\,\vet$ denote the unit tangent vectors in the increasing directions of $\lgt$ and $\lat$. Then, at point $p$ that is away from the poles,\[\pap={\sin\lat}\vep,\qquad\qquad \pat=\vet,\]namely,
\be\label{ON}{1\over\sin\lat} \pap\mbox{\quad and \quad}\pat\mbox{\qquad form an orthonromal basis of\;\;}TM_p\ee
 Therefore, the musical isomorphisms, in $\lgt,\lat$ coordinates, satisfy \[({1\over\sin\lat} \pap)^\flat={\sin\lat} d\lgt\mbox{\quad and \quad}(\pat)^\flat=d\lat\mbox{\quad form an orthonomal basis of\;\;}T^*M_p.\]

In this context, the Hodge *-operator satisfies
\[\begin{split}\mbox{for 1-forms,\qquad}& *({A_1 } d\lgt+A_2d\lat)={A_1\over\sin\lat} d\lat-{A_2\sin\lat} d\lgt\\
   \mbox{for 0-forms and 2-forms,\qquad}& *({A} d\lgt \wedge d\lat)={A\over\sin\lat},\qquad *A={A\sin\lat} d\lgt \wedge d\lat
\end{split}\] 

The differential operators defined in \eqref{def:div} --- \eqref{def:rgrad} then become,
\begin{align*}\mbox{for vector field }\vu=&u_1\vep+u_2\vet\\
\dv\vu=&{1\over\sin\theta}(\pap u_1+\pat(u_2\sin\theta))\\
\curl\vu=&{1\over\sin\theta}(\pap u_2-\pat(u_1\sin\theta))\end{align*}
and\begin{align*}\mbox{for scalar field }f\;\;&\\
\grad f=&{1\over\sin\theta}\pap f\vep +\pat f\vet\\
\rgrad f=&\pat f\vep-{1\over\sin\theta}\pap f\vet\\
\Delta f=&{1\over \sin^2\theta}(\pap^2 f+\sin\theta\pat(\sin\theta\pat f))\end{align*}

\section{\bf Appendix B: Uniform Estimates  Independent of $\ep$}
In this section, we use energy methods to prove local-in-time existence of  classical solutions for the incompressible Euler equations {\it independent of } the Rossby number $\ep$. Rewrite the equation as in \eqref{Euler:cL}, \be\label{Euler:ap}\pt\vu+\rgrad\Deltainv \curl(\nabla_\vu\vu)={1\over\ep}\cL[\vu],\ee
where operator $\cL$, as in \eqref{def:cL}, is defined by \be\label{def:cL:ap}\cL[\vu]:=\rgrad\Deltainv \curl(z\vup).\ee 

The main challenge rises from the nontrivial geometry of $S^2$: only a selective set of differential-integral operators on $S^2$ commute with each other. Although this is not a problem regarding well-posedness with fixed $\ep$, it causes difficulties in obtaining $\ep$-independent estimates. The fact that our $\cL$ has variable coefficients adds another layer of difficulties.  In proving the following theorem, we will address these commutability issues specifically.

\begin{theorem}\label{thm:uniform}
Consider the  incompressible Euler equations \eqref{Euler:ap}, \eqref{def:cL:ap} on a rotating sphere $S^2$ with div-free initial data $\vu_0$. Given any integer $k>2$, assume  $\vu_0\in H^k(S^2)$. Then, there exists universal constants $C_0,T_0$ independent of $\ep$ so that
\[\|\vu(t,\cdot)\|_{H^k}\le C_0\|\vu_0\|_{H^k}\mbox{\quad for any }t\in \left[0,{T_0\over\|\vu_0\|_{H^k}}\right].\]
\end{theorem}
\begin{proof}For simplicity, we only prove the case when $k$ is even. 

First, we show that for any $\vu\in L^2(S^2)$,
\be\label{L2:ortho}\int_{S^2}\vu\cdot\cL[\vu]=0 \text{\quad if  $\dv\vu=0$}.\ee
Indeed, by definition \eqref{def:cL:ap} and Hodge decompositon \eqref{Hodge:vec}, \eqref{irr:inc}, we have
\[z\vup-\cL[\vu]=(\id-\rgrad\Deltainv\curl)[z\vup]=\grad\Delta\dv[z\vup]\]
which is curl-free and therefore $L^2$-orthogonal to a div-free flow $\vu$, namely,
\[\int_{S^2}\vu\cdot(z\vup-\cL[\vu])=0.\] Thus,
\[\int_{S^2}\vu\cdot\cL[\vu]=\int_{S^2}\vu\cdot(z\vup)=0.\]

Secondly, we show that $\Delta$ and $\cL$ commute. Indeed, for any incompressible flow $\vu=\rgrad\Psi$
 \begin{align*}\Delta\cL[\vu]&=\Delta\rgrad\Deltainv\pa_\phi\Psi&\text{by \eqref{cL:stream}}\\
&=\rgrad\Delta\Deltainv\pa_\phi\Psi&\text{by \eqref{surface:L}, \eqref{def:rgrad}} \\
&=\rgrad\Deltainv\Delta \pa_\phi\Psi& \end{align*}
The key step remaining is to show that $\Delta$ and $\pa_\phi$ commute. This can be done using the fact that spherical harmonics $Y_l^m$ are eigenfunctions for both $\Delta$ and $\pa_\phi$. More specifically, for any spherical harmonic  $Y_l^m$,
\[\pa_\phi Y_l^m=\pa_\phi e^{im\phi}Q_l^m(\cos\theta)=imY_l^m\]and therefore 
\[\Delta\pa_\phi Y_l^m=\Delta(imY_l^m)=-iml(l+1)Y_l^m=\pa_\phi\Delta Y_l^m.\]

Lastly, once the commutability of $\Delta$ and $\cL$ are established as above, we easily obtain for even integer $k\ge0$,
\[ \int_{S^2}\Delta^{k/2}\vu\cdot\Delta^{k/2}\cL[\vu]=\int_{S^2}\Delta^{k/2}\vu\cdot\cL[\Delta^{k/2}\vu]=0\]
where the second equality is due to \eqref{L2:ortho} and $\Delta^{k/2}\vu$ also being incompressible (note:  $\Delta$ and $\dv$ commute). Now, take the $L^2$ inner product of $\Delta^{k}\vu$ with both sides of \eqref{Euler:ap}, knowing that the RHS should vanish,
\[\int_{S^2}\Delta^k\vu\cdot\pt\vu+\Delta^k\vu\cdot\rgrad\Delta^{-1}\curl(\nabla_\vu\vu)=0\] and invoke the standard energy methods (e.g. \cite{Majda:book}) to arrive at conclusion, which is clearly $\ep$-independent. \end{proof}
\section{\bf Acknowledgments}
The work of the second author is sponsored in part by AFOSR contract
FA9550-08-1-0055. 

\end{document}